\documentclass{article}
\usepackage{amsfonts, amsmath, amssymb, amsthm, xcolor, mathtools, arydshln,mleftright}
\mleftright % For improved spacing around \left( and \right)
\usepackage[shortlabels]{enumitem}
\usepackage{hyperref}

\usepackage[margin=1.05in]{geometry}

\newtheorem{theorem}{Theorem}[section]
\newtheorem{lemma}[theorem]{Lemma}

\newtheorem{prop}[theorem]{Proposition}
\newtheorem{corollary}[theorem]{Corollary}
\newtheorem{conj}[theorem]{Conjecture}

\theoremstyle{definition}

\theoremstyle{remark}
\newtheorem{remark}[theorem]{Remark}

\DeclareMathOperator{\disc}{disc}

\DeclareMathOperator{\Res}{Res}

\newcommand{\CC}{\mathbb{C}}

\newcommand{\QQ}{\mathbb{Q}}
\newcommand{\PP}{\mathbb{P}}

\newcommand{\ZZ}{\mathbb{Z}}

\newcommand{\floor}[1]{\left\lfloor #1 \right\rfloor}

\newcommand{\ignore}[1]{}
\newcommand{\ds}{\displaystyle}

\renewcommand{\epsilon}{\varepsilon}

\title{The structure of the double discriminant}
\author{Theresa C. Anderson, Ufuoma V. Asarhasa, Adam Bertelli, Fabian Gundlach,\\and Evan M. O'Dorney}
\date{}

\begin{document}

\maketitle

\begin{abstract}
  For a polynomial $f(x) = \sum_{i=0}^n a_i x^i$, we study the double discriminant $DD_{n,k} = \disc_{a_k} \disc_x f(x)$.  This object has been well studied in algebraic geometry, but has been brought to recent prominence in number theory by its key role in the proof of the Bhargava--van der Waerden theorem.  We bridge the knowledge gap for this object by proving an explicit factorization: $DD_{n,k}$ is the product of a square, a cube, and possibly a linear monomial.  Our proof is entirely algebraic.  We also investigate other aspects of this factorization.
\end{abstract}

\section{Introduction}
The discriminant of a polynomial is a fundamental object in mathematics, and the \emph{double discriminant} is an essential tool in the recent Bhargava--van der Waerden theorem.  However, previous investigations of this crucial double discriminant have been from an algebro-geometric lens.  We bridge this gap by providing the first explicit factorization of the double discriminant from a number-theoretic perspective.  Given the recent breakthroughs using this object, we envision this factorization and its significance to be useful in other pertinent problems in number theory. 

Recall that the \emph{discriminant} of a polynomial $f(x) = \sum_{i=0}^n a_i x^i$ is a polynomial $D = D_n(a_n,\ldots,a_0) = \disc f \in \ZZ[a_0,\ldots,a_n]$ that vanishes precisely when $f$ has a multiple root. In his groundbreaking work on van der Waerden's conjecture, Bhargava \cite{Bhargava_vdW_short,Bhargava_vdW} makes use of a \emph{double discriminant,} the discriminant $DD_{n,k} = DD_{n,k}(a_n, \ldots, \widehat{a_k}, \ldots a_0)$ of the discriminant $D$ with respect to one of the coefficients $a_k$. Since his proof does not need any properties of $DD_{n,k}$ beyond it being a nonzero polynomial, the double discriminant defined above remains mostly unstudied from a number-theoretic perspective.  However, given this novel, recent use, pertinent unanswered questions about the structure of $DD_{n,k}$ arise.  This work answers several of these questions.

As previously mentioned, while the double discriminant has been studied before from other perspectives, much still remains unknown.  For instance, the book \cite{GKZ}, an oft-cited reference on discriminants, does not deal with it. Previous work abounds from a purely algebraic lens: see the work of Lazard--McCallum \cite{LM09} and the references therein (\cite{Henrici1866,BM09,Han16}, among others). These works study the double discriminant $DD = \disc_y \disc_x f(x,y)$ of a general bivariate form $f(x,y)$. Notably, this double discriminant turns out to have square and cube factors, which also appear in our setting. (This is in contrast to the discriminant $D$, which is well known to be irreducible.) These square and cube factors are important as they correspond to two essentially different ways that (not $f$ itself but) $\disc_x f$ can have multiple roots: if $f$ has two double roots or one triple root (see Theorem \ref{thm:triple_double}). Unfortunately, these square and cube factors from the literature  are described in terms of elimination ideals or Macaulay resultants, which are difficult to work with.  We overcome this difficulty by making this factorization explicit.  

\begin{theorem}
\label{thm:DD_fzn_intro}
  For all $n \geq 2$, $0\leq k \leq n$, there is a factorization in $\ZZ[a_0,\ldots,\widehat{a_k}, \ldots,a_n]$, 
  \begin{equation}  
   DD_{n,k} = c_{n,k} R_0R_nA_{n,k}^3 B_{n,k}^2  
  \end{equation}
  where $c_{n,k} \in \ZZ$ is a constant, $A_{n,k}, B_{n,k}$ can be described explicitly, up to scaling (see Theorems \ref{thm:DDn0_fzn} and \ref{thm:factorization-nonzero-k} below), and
	\[
		R_0 :=
		\begin{cases}
			a_0 &\textnormal{if }k=1,\\
			1 &\textnormal{otherwise},
		\end{cases}
		\qquad\textnormal{and}\qquad
		R_n :=
		\begin{cases}
			a_n &\textnormal{if }k=n-1,\\
			1 &\textnormal{otherwise}.
		\end{cases}
	\]
\end{theorem}
\begin{remark}  
For $n=2$ we have $DD_{2,0} = DD_{2,2} = 1$ and $DD_{2,1} = 16a_0a_2$.
    \end{remark}
Note that when $n=2$, $f$ cannot have a pair of double roots nor a triple root, hence there are no square nor cube factors in the factorization.  When $n=3$ only a triple root is possible, so the square term $B_{3,k}$ disappears, and $DD_{3,k}$ is a cube up to scaling.

Our factorization is not only useful from a number-theoretic perspective, but its proof, avoiding algebraic geometry and excessive computations, also illuminates this key structure.  Our techniques are mostly elementary without being heavily computational or relying on previous work.

Secondly, we also investigate the outlying constant $c_{n,k}$, the GCD of the coefficients of $DD_{n,k}$. As $n$ grows, $c_{n,k}$ tends to become larger but remains smooth (that is, a product of small primes). Specifically:
\begin{theorem}
  We have $2^{n-1} \mid c_{n,0}$, and for $k \geq 1$, we have $2^{n} \mid c_{n,k}$.
\end{theorem}

\begin{conj}[see Conjecture \ref{conj:c_full} below for a more detailed conjecture]
  The primes dividing $c_{n,k}$, for $1 \leq k \leq n-1$, are precisely those dividing $2\gcd(n,k)$.
\end{conj}
These primes may appear to high powers; see Table \ref{tab:c} for numerical data.

This paper is organized as follows.  Section \ref{section:properties} gives relevant information on discriminants, resultants and double discriminants.  This is followed by Section \ref{section:factorization}, where our main result, Theorem \ref{thm:DD_fzn_intro}, is carefully proven via several steps. In Section \ref{sec:signif_AB}, we relate the factors $A$ and $B$ in Theorem \ref{thm:DD_fzn_intro} to triple and double roots of $f$. Finally, we discuss the constants appearing in our factorization in Section \ref{section:constant}.

\subsection{Acknowledgments}
TCA was supported by the NSF under DMS-2231990 and CAREER DMS-2237937.
FG was supported by the Deutsche Forschungsgemeinschaft (DFG, German Research Foundation) --- Project-ID 491392403 --- TRR 358 (Project A4).
The authors thank Hongyi (Brian) Hu for helpful discussions.

\section{Key properties of discriminants, resultants, and double discriminants}
\label{section:properties}
\subsection{Discriminants and resultants}
The discriminant $D = D_n \in \ZZ[a_0,\ldots,a_n]$ of a polynomial $f(x)$ and the resultant $\Res(f(x), g(x))$ of two polynomials $f(x)$ and $g(x)$ are well studied. There are two main ways to define them. The first is to stipulate that the discriminant of a completely factored polynomial
\[
  f(x) = a_n(x - r_1) \cdots (x - r_n)
\]
is
\begin{equation} \label{eq:disc_from_roots}
  D = a_{n}^{2n-2} \prod_{1 \leq i < j \leq n} (r_i - r_j)^2
\end{equation}
and the resultant of $f(x)$ and a polynomial $g(x)$ of degree $m$ is
\begin{equation}\label{eq:res_from_roots}
	\Res(f(x), g(x)) = a_n^m \prod_{i=1}^n g(r_i).
\end{equation}
They are related by
\begin{equation}\label{eq:disc_from_res}
	D = \frac{(-1)^{n(n-1)/2}}{a_n} \Res(f(x), f'(x)).
\end{equation}

They can also be defined by a Sylvester determinant, here shown only for the discriminant:
\begin{equation} \label{eq:Sylv}
  D = \frac{(-1)^{n(n-1)/2}}{a_n} \begin{vmatrix}
      a_{n} & a_{n-1} & \cdots & a_2 & a_1 & a_0 & 0 & \cdots & 0 \\
      0 & a_n & \cdots & a_3 & a_2 & a_1 & a_0 & \ddots & 0  \\
      \vdots & \ddots & \ddots & \ddots& \ddots& \ddots& \ddots & \ddots & 0  \\
      0 & \cdots & 0 & a_n & a_{n-1} & a_{n-2} & a_{n-3} & \cdots & a_0  \\
      n a_n & (n-1)a_{n-1} & \cdots & 2a_2 & a_1 & 0 & 0 & \cdots & 0 \\
      0 & n a_n & \cdots & 3 a_3 & 2a_2 & a_1 & 0 & \cdots & 0 \\
      \vdots & \ddots & \ddots & \ddots& \ddots& \ddots & \ddots & \ddots & \vdots  \\
      0 & 0 & \cdots & na_n & \cdots & \cdots & 2a_2 & a_1 & 0  \\
      0 & 0 & \cdots & 0 & na_n & \cdots & \cdots & 2a_2 & a_1  \\
  \end{vmatrix}
\end{equation}
Both definitions are equivalent \cite[Ch.\ 12, (1.23)]{GKZ}. (Note that \cite{GKZ} uses a definition of discriminant differing by a sign from the modern one used here.) Here are a few well-known properties of use to us:

\begin{prop}[\cite{GKZ}, Ch.\ 12, (1.24)]
   \label{prop:D_bihom}
  The discriminant $D_n$ is homogeneous and quasi-homogeneous in the coefficients; that is, each term 
  \[
    c_{e_{0},e_1,\ldots,e_n} a_0^{e_0}a_1^{e_1}\cdots a_n^{e_n}
  \]
  satisfies
  \[
    \sum_i e_i = 2n - 2, \quad \sum_i ie_i = n(n-1).
  \]
\end{prop}

\begin{prop}[a special case of \cite{GKZ}, Ch.\ 12, (1.26)] \label{prop:D_sym}
  We have the symmetry
  \[
    D(a_0,\ldots, a_n) = D(a_n,\ldots, a_0).
  \]
\end{prop}

\begin{prop}[\cite{GKZ}, Ch.\ 12, (1.4) and (1.5)]\label{prop:res-properties}~
	\begin{enumerate}[(a)]
	\item\label{item:res-properties-swap}
		$\Res(f(x), g(x)) = (-1)^{nm} \Res(g(x), f(x))$ if $f$ and $g$ have degrees~$n$ and~$m$, respectively.
	\item\label{item:res-properties-mult}
		$\Res(f_1(x) f_2(x), g(x)) = \Res(f_1(x), g(x)) \Res(f_2(x), g(x))$.
    \item\label{item:disc-mult}
        $\disc(f \cdot g)=\disc(f)\cdot\Res(f,g)^2\cdot \disc(g)$.
	\item\label{item:res-properties-x}
		$\Res(x,g(x)) = g(0)$.
	\end{enumerate}
\end{prop}

\begin{prop} \label{prop:ldg}
  If we view $D$ as a polynomial in exactly one coefficient $a_k$, then
  \[\deg_{a_k}(D)=\begin{cases}
    n-1, & k \in \{0, n\}\\
    n, &  0 < k < n.
  \end{cases}\]
  Moreover, the leading term of $D$ is   \[
    \begin{cases}
        \ds (-1)^{\frac{n(n-1)}{2}} n^n a_0^{n-1} a_n^{n-1}, & k \in \{0, n\}\\
        \ds (-1)^{\frac{n(n-1)}{2}-k(n-k)} \frac{(ka_0)^k ((n-k)a_n)^{n-k}}{a_0 a_n}\cdot a_k^n, & 0 < k < n.
    \end{cases}
  \]
\end{prop}
This is presumably well-known, but see Lemmas \ref{lem:disc-zero} and \ref{lem:disc-nonzero-k} for the computations.

\subsection{Double discriminants}
\begin{lemma} \label{lem:deg} For $0 \leq k \leq n$, the double discriminant
    $DD_{n,k}$ is a homogeneous and quasi-homogeneous polynomial in $\{a_0,\ldots,a_n\}\setminus\{a_k\}$; specifically, each term is of the form
    \[
      c_{e_0,\ldots, \widehat{e_k}, \ldots, e_n} \prod_{\substack{0 \leq i \leq n \\ i \neq k}} a_i^{e_i}
    \]
    with
    \[
      \sum_i e_i = \begin{cases}
        (3n-6)(n-1), & k \in \{0, n\}\\
        (3n-4)(n-1), & 0 < k < n
      \end{cases}
    \]
    and
    \[
      \sum_i i e_i = \begin{cases}
        n(n-1)(2n - 4), & k = 0 \\
        n(n-1)(2n - k - 2), & 0 < k \leq n.
      \end{cases}
    \]
\end{lemma} 
\begin{proof}

For brevity, we treat only the case $0 < k < n$, the others being analogous.
We have
\[D=a_k^nc_n+a_k^{n-1}c_{n-1}+\cdots+c_0\]
for polynomials $c_j(a_0,\ldots,a_{k-1},a_{k+1},\ldots,a_n)$ of homogeneous degree $2n-2-j$ and quasi-homogeneous degree $n(n-1) - jk$. A generic term in $DD_{n,k}$ is of the form $\prod_{j=0}^n c_j^{d_j}$, where $\sum_{j=0}^n d_j=2n-2$ and $\sum_{j=0}^n jd_j=n(n-1)$. Thus, the homogeneous degree in all the $a_k$'s of each term of $DD_{n,k}$ is
\begin{align*}
  \deg DD_{n,k} &= \sum_{j=0}^{n}(2n-2-j)d_j \\
  &= (2n-2)\sum_{j=0}^{n} d_j-\sum_{j=0}^{n-1}jd_j \\
  &= (2n-2)(2n-2)-n(n-1) = (3n-4)(n-1),
\end{align*}
and the quasi-homogeneous degree is
\begin{align*}
  \deg_{\text{quasi-hom}} DD_{n,k} &= \sum_{j=0}^{n}(n(n-1) - jk)d_j \\
  &= n(n-1)\sum_{j=0}^{n} d_j-k\sum_{j=0}^{n}jd_j \\
  &= n(n-1)(2n-2)-n(n-1)k = n(n-1)(2n - k - 2),
\end{align*}
as claimed.
\end{proof}

\begin{lemma} \label{lem:sym}
    $DD_{n,n-k}(a_0, \ldots, \widehat{a_k}, \ldots, a_n) = 
  DD_{n,k}(a_n, \ldots, \widehat{a_k}, \ldots a_0)$.
\end{lemma}

\begin{proof}
This follows by taking the discriminant of both sides of Proposition \ref{prop:D_sym} with respect to $a_k$.
\end{proof}

In view of this, we hereafter only need to consider those $DD_{n,k}$ with $k \leq n/2$.

\section{Main result: factorization of the double discriminant}
\label{section:factorization}
In this section we prove the striking factorization patterns of the double discriminant in terms of the powers $A_{n,k}^3$ and $B_{n,k}^2$ to prove Theorem \ref{thm:factorization-nonzero-k}.  While it may be possible to derive the square-fullness of $DD_{n,k}$ from the work of Lazard--McCallum (\cite[Theorem 1]{LM09}, which in turn is a corollary of Bus{\'e}--Mourrain \cite[Theorem 6.8]{BM09}), this approach appears both technical and less natural from a number-theoretic perspective.  In particular, Lazard--McCallum give a factorization
\begin{equation}
  \disc_y \disc_x f \sim PP^\infty QQ^\infty RR^\infty S
\end{equation}
of the double discriminant of a general bivariate polynomial $f = f(x,y)$. Here the symbol ``$\sim$'' means that either both sides are $0$, or they have the same irreducible factors. Moreover, the irreducible factors of $QQ^\infty$ and of $RR^\infty$ appear with multiplicities at least $2$ and $3$, respectively. If we can show that $P = P^\infty = S = 1$, then we can deduce that any irreducible factor of $DD_{n,k}$ appears to an exponent of at least $2$. Due to difficulties in explicitly computing the elimination ideals in \cite{LM09} as well as desiring a more explicit factorization, we approached the question from scratch to use algebraic (and number-theoretic) techniques.  This lets us provide insight on the factors in our factorization.

\subsection{The case \texorpdfstring{$k = 0$}{k = 0}}\label{sn:k0}

We now establish the complete factorization in the case $k=0$ (and therefore in the case $k=n$).  We do this case separately only to avoid some very minor technical details and for clarity.

\begin{theorem} \label{thm:DDn0_fzn}
    We have a factorization
    \begin{equation} \label{eq:DDn0_fzn}
      DD_{n,0} = A_{n,0}^3\cdot B_{n,0}^2
    \end{equation}
    in $\QQ[a_1,\ldots, a_n]$, where $A_{n,0}=\disc(f')$, and
    \begin{equation}
        B_{n,0}=n^{n-2} a_n^{(n-2)(n-4)}\prod_{i<j}\dfrac{f(s_i)-f(s_j)}{(s_i-s_j)^3}, \label{eq:B_{n,0}}
    \end{equation}
    where $s_1,\ldots,s_{n-1}$ are the roots of $f'$.
\end{theorem}

\begin{remark}
This verifies Theorem \ref{thm:DD_fzn_intro} in the case $k = 0$. By Gauss's lemma for polynomials, this factorization lifts to $\ZZ[a_1,\ldots,a_n]$ after possible rescaling, as $A_{n,0}$ may be divisible by an integer greater than $1$ and $B_{n,0}$ may have non-integer coefficients.
\end{remark}

Write $f(x) = f_0(x) + a_0$. To prove Theorem \ref{thm:DDn0_fzn}, we begin with a lemma expressing $DD_{n,0}$ in terms of the roots of $f' = f_0'$:

\begin{lemma}
\label{lem:disc-zero}
Let the roots of $f_0'$ be $s_1,\ldots, s_{n-1}$ (with multiplicity). We have the identities
    \begin{equation} \label{eq:D_from_s_i}
      D = \disc(f) = (-1)^{n(n-1)/2} n^n a_n^{n-1} \prod_{i} \big(a_0 + f_0(s_i)\big)
    \end{equation}
and
\begin{equation} \label{eq:DDn0}
    DD_{n,0}=\disc_{a_0}(\disc(f))= n^{2n(n-2)} a_n^{2(n-1)(n-2)}\prod_{i<j}\big(f_0(s_i)-f_0(s_j)\big)^2.
\end{equation}
\end{lemma}

\begin{proof}
    Since $f'(x) = na_nx^{n-1} + \cdots + a_1$, we have
    \begin{align*}
        D
        &= (-1)^{n(n-1)/2} a_n^{-1} \Res(f(x),f'(x))
        &&\textnormal{by Equation~(\ref{eq:disc_from_res})}\\
        &= (-1)^{n(n-1)/2} a_n^{-1} (n a_n)^n \prod_i f(s_i)
        &&\textnormal{by Equation~(\ref{eq:res_from_roots}) and Proposition~\ref{prop:res-properties}\ref{item:res-properties-swap}}\\
        &= (-1)^{n(n-1)/2} n^n a_n^{n-1} \prod_i (a_0 + f_0(s_i)).
    \end{align*}
    Taking the discriminant using \eqref{eq:disc_from_roots} yields
    \begin{align*}
        DD_{n,0}&= \big((-1)^{n(n-1)/2}n^n a_n^{n-1}\big)^{2(n-1)} \prod_{i<j}\big(f_0(s_i)-f_0(s_j)\big)^2 \\
        &= n^{2n(n-2)} a_n^{2(n-1)(n-2)}\prod_{i<j}\big(f_0(s_i)-f_0(s_j)\big)^2,
    \end{align*}
    as desired.
\end{proof}

\begin{proof}[Proof of Theorem \ref{thm:DDn0_fzn}]
    It follows from the preceding lemma that an identity of the shape \eqref{eq:DDn0_fzn} holds, where $C = n^{2n(n-2)}$ and $A_{n,0}, B_{n,0} \in \QQ[a_n, s_1, \ldots, s_{n-1}]$ are as defined in the theorem. Indeed, one checks that $A_{n,0}$ and $B_{n,0}$ are symmetric in the $s_i$'s and invariant under adding a constant to $f$, and thus they belong to the rational function field $\QQ(a_1,\ldots,a_n)$. However, it is not obvious that $B_{n,0}$ is a polynomial.
    
    First, we show that $(s_i-s_j)^3\mid f_0(s_i)-f_0(s_j)$ in $\QQ[a_n, s_1, \ldots, s_{n-1}]$. Write $f_0(s_i)-f_0(s_j) = \int_{s_j}^{s_i} f_0'(x)\, dx$. Observe that $f_0'(x)=(x-s_i)(x-s_j)g(x)$ for some polynomial $g \in \QQ[a_n,s_1,\ldots, s_{n-1},x]$. Hence, if we substitute $t=x-s_j$ and let $h=s_i-s_j$, we get
    \begin{equation}
      f_0(s_i)-f_0(s_j) = \int_0^h t(t-h)g(t+s_j)\, dt
    \end{equation}
    Note that each term in the integrand has either a $t^2$ or an $ht$ factor, so when we integrate, each term in the result will have either a $t^3$ or an $ht^2$ factor, so taking the difference from $h$ to $0$ results in an $h^3$ factor as desired. Observe that $A_{n,0}=(na_n)^{2(n-2)}\prod\limits_{i<j}(s_i-s_j)^2$, hence we have $A_{n,0}^3\mid DD_{n,0}$ in $\QQ[a_n, s_1, \ldots, s_{n-1}]$.

    Therefore $B_{n,0} \in \QQ[a_n, s_1, \ldots, s_{n-1}]$ is a polynomial symmetric in the $s_i$'s. Hence $B_{n,0}$ is a polynomial in $a_n$ and the elementary symmetric functions
    \[
      \sum_{i_1 < \cdots < i_r} s_{i_1} \cdots s_{i_r} = (-1)^{n-r}\frac{ra_r}{na_n}.
    \]
    So $B_{n,0}$ has no denominator except possibly a power of $a_n$.
    But $B_{n,0}$ cannot have $a_n$ in the denominator either since $DD_{n,0} = A_{n,0}^3 \cdot B_{n,0}^2$ is a polynomial and $A_{n,0} = \disc(f')$ is a polynomial not divisible by $a_n$ (even if we plug in $a_n=0$, the polynomial $f'$ can still have non-zero degree $n-1$ discriminant as the binary form $(n-1)a_{n-1}x^{n-2}y + \cdots + a_1y^{n-1}$ can have $n$ distinct roots in $\PP^1$).
\end{proof}

Note that $A_{n,0}$ is irreducible for $n\geq3$ (and constant for $n=2$) because it is the discriminant of a polynomial $f'$ with independent indeterminate coefficients. We conjecture that $B_{n,0}$ is irreducible as well, indeed that $\dfrac{f_0(s_i)-f_0(s_j)}{(s_i-s_j)^3}$ is irreducible as a polynomial in the $s$'s and $a_n$.

\subsection{The case \texorpdfstring{$1\leq k\leq n-1$}{1 ≤ k ≤ n − 1}}
We can now complete the full range of cases.
\begin{theorem}\label{thm:factorization-nonzero-k}
	We have
	\[
		DD_{n,k} = R_0 R_n A^3 B^2
	\]
	for some (concrete) polynomials $A,B\in\QQ[a_0,\dots,\widehat{a_k},\dots,a_n]$,
	\[
		R_0 :=
		\begin{cases}
			a_0 &\textnormal{if }k=1,\\
			1 &\textnormal{otherwise},
		\end{cases}
		\qquad\textnormal{and}\qquad
		R_n :=
		\begin{cases}
			a_n &\textnormal{if }k=n-1,\\
			1 &\textnormal{otherwise}.
		\end{cases}
	\]
\end{theorem}

We first show why the anomalous monomial factor $a_0$ in the case $k = 1$ (and similarly $a_n$ in the case $k=n-1$) appears.
\begin{lemma}
    \label{lem:a0_Dn1}
    $a_0 \mid DD_{n,1}$.
\end{lemma} 
\begin{proof}
    Let $f=gx+a_0$. Using the product formula for discriminants (Proposition \ref{prop:res-properties}\ref{item:disc-mult}),
    \[\disc(gx)=\disc(g)\cdot\Res(g,x)^2\cdot \disc(x)=\disc(g)\cdot g(0)^2 \cdot 1 = \disc(g) \cdot a_1^2.\]
    Note that $\disc(g x)=D(0,a_1,\ldots,a_n)$, so in particular, every term in $D$ is divisible either by $a_0$ or by $a_1^2$. This in turn means that every term in $\frac{\partial}{\partial a_1}D$ is divisible by $a_0$ or $a_1$, meaning $D$ and $\frac{\partial}{\partial a_1}D$ both vanish when $a_0=a_1=0$, and hence $DD_{n,1}$ vanishes whenever $a_0=0$, establishing $a_0 \mid DD_{n,1}$.
\end{proof}

Let $1\leq k\leq n-1$. Write $f(x) = a_kx^k + f_0(x)$. In place of the polynomial $f'(x) = f_0'(x)$ in the case $k=0$ (see section~\ref{sn:k0}), we now consider the polynomial
\[
	g(x)
	= x f'(x) - k f(x)
 	= x f_0'(x) - k f_0(x)
	= \sum_{\substack{0\leq i\leq n\\ i\neq k}} (i-k)a_ix^i
	= (n-k) a_n x^n + \dots + (-k) a_0.
\]
It is related to the Laurent polynomials $F(x) := \frac{f(x)}{x^k}$ and $F_0(x) := \frac{f_0(x)}{x^k}$ by
\[
	g(x) = x^{k+1} F'(x) = x^{k+1} F_0'(x).
\]
Let $s_1,\dots,s_n$ be the (nonzero!)\ roots of $g(x)$ in the algebraic closure of $\QQ(a_0,\dots,\widehat{a_k},\dots,a_n)$ (a priori counted with multiplicity, but we will actually show that they are distinct).

\begin{lemma}\label{lem:disc-nonzero-k}
	We have
	\[
		D = \disc(f)
		= (-1)^{\frac{n(n-1)}{2}-k(n-k)} \frac{(ka_0)^k ((n-k)a_n)^{n-k}}{a_0 a_n} \prod_{i=1}^n \left(a_k + F_0(s_i)\right).
	\]
\end{lemma}
\begin{proof}
	Noting that $g(x) = (n-k)a_n (x-s_1)\cdots(x-s_n)$ and in particular $-ka_0 = (-1)^n (n-k)a_ns_1\cdots s_n$, we compute
	\allowdisplaybreaks
	\begin{align*}
		D
		&= (-1)^{\frac{n(n-1)}{2}} \frac{1}{a_n} \Res(f(x), f'(x))
		&&\textnormal{by Equation~(\ref{eq:disc_from_res})}\\
		&= (-1)^{\frac{n(n+1)}{2}} \frac{1}{a_0 a_n} \Res(f(x), x f'(x))
		&&\textnormal{by Proposition~\ref{prop:res-properties}\ref{item:res-properties-swap}\ref{item:res-properties-mult}\ref{item:res-properties-x}}\\
		&= (-1)^{\frac{n(n+1)}{2}} \frac{1}{a_0 a_n} \Res(f(x), x f'(x)-k f(x))
		&&\textnormal{by Equation~(\ref{eq:res_from_roots})}\\
		&= (-1)^{\frac{n(n-1)}{2}} \frac{((n-k)a_n)^n}{a_0 a_n} \prod_{i=1}^n f(s_i)
		&&\textnormal{by Proposition~\ref{prop:res-properties}\ref{item:res-properties-swap} and Equation~(\ref{eq:res_from_roots})}\\
		&= (-1)^{\frac{n(n-1)}{2}} \frac{((n-k)a_n)^n}{a_0 a_n} \prod_{i=1}^n (a_k s_i^k + f_0(s_i))\\
		&= (-1)^{\frac{n(n-1)}{2}-(n+1)k} \frac{(ka_0)^k ((n-k)a_n)^{n-k}}{a_0 a_n} \prod_{i=1}^n \left(a_k + \frac{f_0(s_i)}{s_i^k}\right)\\
		&= (-1)^{\frac{n(n-1)}{2}-k(n-k)} \frac{(ka_0)^k ((n-k)a_n)^{n-k}}{a_0 a_n} \prod_{i=1}^n \left(a_k + F_0(s_i)\right).&&
	\qedhere
	\end{align*}
\end{proof}

\begin{corollary}
\label{cor:DDnk F0}
	We have
	\[
		\disc(g) = ((n-k)a_n)^{2n-2} \prod_{i<j} (s_i-s_j)^2
	\]
	and
	\[
		DD_{n,k} = \disc_{a_k}(\disc(f)) = \left(\frac{(ka_0)^k ((n-k)a_n)^{n-k}}{a_0 a_n}\right)^{2n-2} \prod_{i<j} \left(F_0(s_i) - F_0(s_j)\right)^2.
	\]
\end{corollary}
\begin{proof}
	The first equality is clear as $g(x)$ has degree $n$, leading coefficient $(n-k)a_n$, and roots $s_1,\dots,s_n$.

	For the second equality, note that the previous lemma implies that $\disc(f)$ as a polynomial in $a_k$ has degree $n$, leading coefficient \[(-1)^{\frac{n(n-1)}{2}-k(n-k)} \frac{((n-k)a_n)^{n-k}(ka_0)^k}{a_0 a_n},\] and roots $-F_0(s_1),\dots,-F_0(s_n)$.
\end{proof}

Thus, $DD_{n,k} = R_0 R_n A^3 B^2$ with
\begin{equation}\label{eq:def-AB-nonzero-k}
	A := \frac{\disc(g)}{R_0 R_n}
	\qquad\textnormal{and}\qquad
	B := R_0 R_n \left(\frac{(ka_0)^k ((n-k)a_n)^{n-k-3}}{a_0 a_n}\right)^{n-1} \prod_{i<j}\frac{F_0(s_i) - F_0(s_j)}{(s_i-s_j)^3}.
\end{equation}
It remains to show that $A$ and $B$ are polynomials in $a_0,\dots,\widehat{a_k},\dots,a_n$. (It is at least clear that $\disc(g)$ is a polynomial.)

\begin{lemma}~
\begin{enumerate}[(a)]
	\item If $2\leq k\leq n-1$, then $a_0\nmid\disc(g)$.
	\item If $k=1$, then $a_0\mid\disc(g)$ but $a_0^2\nmid\disc(g)$.
\end{enumerate}
\end{lemma}
\begin{proof}
\begin{enumerate}[(a)]
\item
	We can plug in rational numbers for $a_0,\dots,a_n$ so that the resulting polynomial $g(x) = \sum_{i=0}^n (i-k) a_i x^i$ becomes $g(x) = x^n - x$.
	(This is possible because the $x^k$-coefficient of the desired polynomial $x^n-x$ is zero since $2\leq k\leq n-1$.) We then have $a_0=0$ but $\disc(g)\neq0$. Hence, $a_0\nmid\disc(g)$.
\item
	By definition, $g(x)$ has $x$-coefficient zero since $k=1$. If we plug in $a_0=0$, then its constant coefficient also becomes zero, so $g(x)$ becomes divisible by $x^2$, and therefore $\disc(g)=0$. Thus, $a_0\mid\disc(g)$.

	To show that $a_0^2\nmid\disc(g)$, we pick elements $a_0,\dots,a_n$ of $\QQ[t]$ so that the resulting polynomial $g(x)$ becomes
	\[
		g(x) =
		\begin{cases}
			x^2 - t^2 &\textnormal{if }n\in\{2,3\},\\
			(x^{n-2}-1)(x^2 - t^2) &\textnormal{if }n\geq4.
		\end{cases}
	\]
	(This is possible because the $x$-coefficient of the desired polynomial is zero.) We then have $t^2\mid a_0$ but $t^4\nmid\disc(g)$. (In the case $n=3$, keep in mind that we are interpreting $g(x)$ as a polynomial of degree $3$ when computing its discriminant!) Hence, $a_0^2\nmid\disc(g)$.
\qedhere
\end{enumerate}
\end{proof}

\begin{corollary}\label{A_cor}
	We have $R_0R_n\mid\disc(g)$, so $A = \disc(g)/R_0R_n$ is a polynomial in $a_0,\dots,\widehat{a_k},\dots,a_n$. Moreover, $a_0,a_n\nmid A$. In particular, $A\neq0$ as an element of $\QQ[a_0,\dots,\widehat{a_k},\dots,a_n]$, so the roots $s_1,\dots,s_n$ of $g(x)$ are distinct.
\end{corollary}
\begin{proof}
	The lemma shows $R_0 \mid \disc(g)$ and $a_0 \nmid A$. By symmetry, $R_n \mid \disc(g)$ and $a_n\nmid A$.
\end{proof}

Since $F_0'(x) = g(x) / x^{k+1} = (n-k)a_n (x-s_1)\cdots(x-s_n)/x^{k+1}$, the following general lemma about Laurent polynomials implies that $\frac{F_0(s_i) - F_0(s_j)}{(s_i-s_j)^3}$ can be written as a polynomial in $a_n$ and the roots $s_1,\dots,s_n$ of $g(x)$.

\begin{lemma}\label{lem:laurent-antiderivative}
	There are symmetric Laurent polynomials $t_\ell(x_1,x_2) \in \QQ[x_1^{\pm1},x_2^{\pm1}]$ for $\ell\in\ZZ$ such that:
	\begin{enumerate}[(a)]
	\item
		For any Laurent polynomial $F(x)\in K[x^{\pm1}]$ with coefficients in a field $K$ of characteristic zero, if its derivative satisfies
		\begin{equation}\label{eq:Fder}
			F'(x) = (x-s_1)(x-s_2)\sum_{\ell\in\ZZ} c_\ell x^{\ell-2}
		\end{equation}
		with $s_1,s_2\in K^\times$ and all but finitely many $c_\ell\in K$ zero, then
		\[
			F(s_1) - F(s_2) = (s_1 - s_2)^3 \sum_{\ell\in\ZZ} c_\ell t_\ell(s_1,s_2).
		\]
	\item
		We have $t_\ell(x_1,x_2) \in \QQ[x_1,x_2]$ for all $\ell\geq1$.
	\item
		We have $t_{-\ell}(x_1,x_2) = x_1^{-2} x_2^{-2}\cdot t_\ell(x_1^{-1},x_2^{-1})$ for all $\ell\in\ZZ$.
	\item
		We have $t_0(x_1,x_2) = -\frac12 x_1^{-1} x_2^{-1}$ and $t_1(x_1,x_2) = t_{-1}(x_1,x_2) = 0$.
	\item
		For $s_1 = s_2 = s \in K^\times$ and for coefficients $c_\ell\in K$ as in (a), we have $\sum\limits_{\ell\in\ZZ} c_\ell t_\ell(s,s) = -\frac16 \sum\limits_{\ell\in\ZZ} c_\ell s^{\ell-2}$.
	\end{enumerate}
\end{lemma}

\begin{proof}
	For $\ell\geq2$, let $G_\ell(x)$ be an antiderivative of the polynomial $(x-s_1)(x-s_2)x^{\ell-2}$.
	Then,
	\begin{equation}\label{eq:Gderpos}
		G_\ell'(x)
		= (x-s_1)(x-s_2)x^{\ell-2}
	\end{equation}
	and
	\[
		G_1(s_1) - G_1(s_2)
		= \int_{s_2}^{s_1} (x-s_1)(x-s_2) x^{\ell-2} dx
		= \int_0^{s_1-s_2} (x-(s_1-s_2)) x (x+s_2)^{\ell-2} dx.
	\]
	As before, each term in the integrand is divisible by $x^2$ or by $x(s_1-s_2)$. Hence, integrating term by term, we can write
	\begin{equation}\label{eq:Gdiffpos}
		G_\ell(s_1) - G_\ell(s_2)
		= (s_1-s_2)^3\cdot t_\ell(s_1,s_2).
	\end{equation}
	with a polynomial $t_\ell(x_1,x_2)\in\QQ[x_1,x_2]$. Since $G_\ell(s_1) - G_\ell(s_2)$ and $(s_1-s_2)^3$ are both antisymmetric, the polynomials $t_\ell(x_1,x_2)$ are symmetric. Since the only terms in the integrand not divisible by $x^3$ or by $x^2(s_1-s_2)$ are $x^2 s_2^{\ell - 2}$ and $-(s_1-s_2)xs_2^{\ell-2}$, we have
	\begin{equation}\label{eq:tells}
		t_\ell(s, s) = \tfrac13 s^{\ell-2} - \tfrac12 s^{\ell-2} = -\tfrac16 s^{\ell-2}.
	\end{equation}
	
	For $\ell\geq2$, letting $H_\ell(x)$ be an antiderivative of the polynomial $(x-s_1^{-1})(x-s_2^{-1})x^{\ell-2}$, we have
	\begin{equation}\label{eq:Gderneg}
		\frac{d}{dx} H_\ell(x^{-1})
		= - (x^{-1} - s_1^{-1})(x^{-1} - s_2^{-1})x^{-\ell+2}x^{-2}
		= - s_1^{-1} s_2^{-1} (x - s_1)(x - s_2) x^{-\ell-2}
	\end{equation}
	and
	\begin{equation}\label{eq:Gdiffneg}
		H_\ell(s_1^{-1}) - H_\ell(s_2^{-1})
		= (s_1^{-1} - s_2^{-1})^3\cdot  t_\ell(s_1^{-1},s_2^{-1})
		= - s_1^{-3} s_2^{-3} (s_1 - s_2)^3\cdot t_\ell(s_1^{-1},s_2^{-1}).
	\end{equation}
	The derivative of any Laurent polynomial $F(x)$ has $x^{-1}$-coefficient zero. Hence, looking at the $x^{-1}$-coefficients in Equation~(\ref{eq:Fder}), we obtain the relation
	\begin{equation}\label{eq:rel}
		0 = c_{-1} - (s_1+s_2) c_0 + s_1 s_2 c_1.
	\end{equation}
	Lastly, we let
	\[
		I(x) :=
		\tfrac12 c_1 x^2
		+ (c_0 - (s_1+s_2) c_1) x
		+ ((s_1+s_2) c_{-1} - s_1 s_2 c_0) x^{-1}
		- \tfrac12 s_1 s_2 c_{-1} x^{-2}.
	\]
	Then,
	\begin{equation}\label{eq:Hder}
	\begin{aligned}
		I'(x)
		&= c_1 x
		+ (c_0 - (s_1 + s_2) c_1)
		- ((s_1+s_2)c_{-1} - s_1s_2c_0)x^{-2}
		+ s_1s_2c_{-1}x^{-3} \\
		&\stackrel{\mathclap{(\ref{eq:rel})}}{=} (x-s_1)(x-s_2)\sum_{\ell\in\{1,0,-1\}} c_\ell x^{\ell-2},
	\end{aligned}
	\end{equation}
	and
	\begin{equation}\label{eq:Hdiff}
	\begin{aligned}
		I(s_1) - I(s_2)
		&= \tfrac12 c_1 (s_1^2 - s_2^2)
		+ (c_0 - (s_1+s_2)c_1) (s_1 - s_2) \\
		&\qquad + ((s_1+s_2)c_{-1} - s_1s_2c_0) (s_1^{-1} - s_2^{-1})
		- \tfrac12 s_1 s_2 c_{-1} (s_1^{-2} - s_2^{-2}) \\
		&= -\tfrac12 c_1 (s_1^2 - s_2^2)
		+ 2 c_0 (s_1 - s_2)
		- \tfrac12 c_{-1} s_1^{-1} s_2^{-1} (s_1^2 - s_2^2) \\
		&\stackrel{\mathclap{(\ref{eq:rel})}}{=}{} 2c_0(s_1-s_2)
		- \tfrac12 c_0 s_1^{-1} s_2^{-1} (s_1^2 - s_2^2) (s_1+s_2) \\
		&= - \tfrac12 c_0 s_1^{-1} s_2^{-1} (s_1 - s_2)^3.
	\end{aligned}
	\end{equation}
	Let
	\[
		J(x) := \sum_{\ell\geq2} c_\ell G_\ell(x) - \sum_{\ell\geq2} c_{-\ell} s_1 s_2 H_\ell(x^{-1}) + I(x).
	\]
	Combining Equations~(\ref{eq:Gderpos}), (\ref{eq:Gderneg}), (\ref{eq:Hder}), we see that $J'(x) = (x-s_1)(x-s_2)\sum_{\ell\in\ZZ} c_\ell x^{\ell-2} = F'(x)$ and therefore $J(s_1)-J(s_2) = F(s_1)-F(s_2)$. Combining Equations~(\ref{eq:Gdiffpos}), (\ref{eq:Gdiffneg}), (\ref{eq:Hdiff}), we then obtain
	\begin{align*}
		F(s_1) - F(s_2)
		&= J(s_1) - J(s_2) \\
		&= (s_1-s_2)^3 \cdot \left(
			\sum_{\ell\geq2} c_\ell t_\ell(s_1,s_2)
			+ \sum_{\ell\geq2} c_{-\ell} s_1^{-2} s_2^{-2} t_\ell(s_1^{-1},s_2^{-1})
			+ c_0 \cdot \left(-\frac12 s_1^{-1} s_2^{-1}\right)
		\right).
	\end{align*}
	This proves parts (a)--(d) of the claim as each $t_\ell(x_1,x_2)$ with $\ell\geq2$ is a symmetric polynomial in $\QQ[x_1,x_2]$. For part (e), we compute:
    \[
	\begin{multlined}
	    \sum_{\ell\geq2} c_\ell t_\ell(s,s) + \sum_{\ell\geq2} c_{-\ell} s^{-4} t_\ell(s^{-1},s^{-1}) + c_0\cdot\left(-\frac12 s^{-2}\right) \\
		\qquad\stackrel{\mathclap{(\ref{eq:tells})}}{=} -\frac16 \left( \sum_{\ell\geq2} c_\ell s^{\ell-2} + \sum_{\ell\geq2} c_{-\ell} s^{-\ell-2} + 3 c_0 s^{-2} \right)
		\stackrel{\mathclap{(\ref{eq:rel})}}{=} -\frac16 \sum_{\ell\in\ZZ} c_\ell s^{\ell-2}.
	\end{multlined}
	\qedhere
    \]
\end{proof}

\begin{proof}[Proof of Theorem~\ref{thm:factorization-nonzero-k}]
	Define $A$ and $B$ as in Equation~(\ref{eq:def-AB-nonzero-k}). According to Corollary~\ref{A_cor}, we have $A\neq0$ and the roots $s_1,\dots,s_n$ are distinct, so $B$ is well-defined. Let $e_p(r_1,\dots,r_q)$ denote the $p$-th elementary symmetric polynomial in $r_1,\dots,r_q$. Recall that
	\[
		F_0'(x) = g(x) / x^{k+1} = (n-k) a_n (x-s_1)\cdots(x-s_n) / x^{k+1}.
	\]
	Thus, for any $i<j$,
	\[
		F_0'(x) = (x-s_i)(x-s_j) \cdot \sum_{p=0}^{n-2} (n-k) a_n (-1)^p e_p(s_1,\dots,\widehat{s_i},\dots,\widehat{s_j},\dots,s_n) x^{n-2-p-k-1}.
	\]
	By Lemma~\ref{lem:laurent-antiderivative}, we then have
	\[
		\frac{F_0(s_i) - F_0(s_j)}{(s_i-s_j)^3}
		= \sum_{p=0}^{n-2} t_{n-p-k-1}(s_i,s_j) (n-k) a_n (-1)^p e_p(s_1,\dots,\widehat{s_i},\dots,\widehat{s_j},\dots,s_n)
	\]
	for some symmetric Laurent polynomials $t_\ell(x_1,x_2)\in\QQ[x_1^{\pm1},x_2^{\pm1}]$. Thus, the product
	\[
		B = R_0 R_n \left(\frac{(ka_0)^k ((n-k)a_n)^{n-k-3}}{a_0 a_n}\right)^{n-1} \prod_{i<j}\frac{F_0(s_i) - F_0(s_j)}{(s_i-s_j)^3}.
	\]
	is a nonzero constant times a power of $a_0$ times a power of $a_n$ times a power of $s_1\cdots s_n = \pm ka_0/(n-k)a_n$ times a symmetric polynomial in $s_1,\dots,s_n$. Possibly up to some factors of $a_0$ and $a_n$ in the denominator, $B$ can therefore be written as a polynomial in the coefficients of $g(x) = (n-k)a_n + \cdots + (-k)a_0 = (n-k)a_n(x-s_1)\cdots(x-s_n)$, or equivalently as a polynomial in the variables $a_0,\dots,\widehat{a_k},\dots,a_n$.

	Since $DD_{n,k} = R_0 R_n A^3 B^2$ is a polynomial and $a_0^2, a_n^2 \nmid R_0 R_n$ and $a_0, a_n \nmid A$, the factor $B$ in fact cannot have $a_0$ or $a_n$ in the denominator. Thus, $B$ is a polynomial.
\end{proof}

\subsection{Significance of the \texorpdfstring{$A$}{A} and \texorpdfstring{$B$}{B} factors}
\label{sec:signif_AB}
The factorization of $DD_{n,k}$ has the following pleasing interpretation in terms of the roots of $f$.
\begin{theorem}\label{thm:triple_double} Let $f$ be a polynomial of degree $n \geq 2$ over a field of characteristic $0$.
  \begin{enumerate}[(a)]
    \item\label{item:triple} If $f$ has a root of multiplicity at least~$3$, then $A_{n,k} = 0$ for $0 \leq k < n$.
    \item\label{item:double} If $f$ has two roots of multiplicity at least~$2$, or one root of multiplicity at least~$4$, then $B_{n,k} = 0$ for $0 \leq k < n$.
  \end{enumerate}
\end{theorem}
\begin{proof}
  \begin{enumerate}[(a)]
    \item By passing to the limit, it suffices to prove this when $a_0 \neq 0$, so that $R_0R_n\neq 0$. Then $f$ is divisible by $(x-s)^3$, $s \neq 0$, so $f'$ and $g(x) = x f' - k f$ are divisible by $(x-s)^2$. Thus $\disc f'$ and $\disc g$ vanish, implying that $A_{n,0} = \disc f' = 0$ and, for $1 \leq k < n$, we have $A_{n,k} = \disc g/(R_0 R_n) = 0$.
    \item We may assume $f$ is divisible by $(x - s)^2(x - t)^2$ with $s \neq t$, since the case $s = t$ follows by passing to the limit as $B_{n,k}$ is a polynomial.
    
    Observe that $f'$ is divisible by $(x - s)(x - t)$, and so we may take $s_1 = s, s_2 = t$ in the factorization of $f'$ (respectively $g$). Then one of the factors of $B_{n,0}$ is
    \[
      \frac{f_0(s_1)-f_0(s_2)}{(s_1-s_2)^3} = \frac{0 - 0}{(s_1-s_2)^3} = 0,
    \]
    and analogously for $B_{n,k}$, $1 \leq k \leq n-1$.
    
    To complete the proof, we must show that no other factor of $B_{n,k}$ has an identically vanishing denominator. For this, we can exhibit a polynomial with two double roots such that $f'$ (respectively $g$) has no multiple roots. For example, if $1 < r_1 < \cdots < r_{n-4} < 2$ are real numbers (note that $n \geq 4$), then
    \[
      f(x) = (x-1)^2(x-r_1)\cdots(x-r_{n-4})(x-2)^2
    \]
    has two double roots, while $f'$ has no multiple roots, since by Rolle's theorem the roots of $f'$ interlace with those of $f$. Moreover, by counting the changes in direction of $F(x)=f(x)/x^k$, we obtain that $g(x) = x^{k+1} F'(x)$ has $n$ distinct real roots as well. Since we also ensured that $a_0$ and $a_n$ are nonzero, the identities \eqref{eq:B_{n,0}} and \eqref{eq:def-AB-nonzero-k} are applicable without division by $0$, and we get $B_{n,k} = 0$. \qedhere
  \end{enumerate}
\end{proof}

The converse is also true in the appropriate level of generality.
\begin{theorem}\label{thm:triple_double_converse}
	Let $f\in\CC[x]$ be a polynomial of degree $n\geq2$, and let $A_{n,k}$, $B_{n,k}$ be the factors of its double discriminant with respect to one coefficient $a_k$, $0 \leq k < n$. If $k\neq0$, assume $f(0)\neq0$.
	\begin{enumerate}[(a)]
	\item
		If $A_{n,k} = 0$, then there is a value of $c\in\CC$ for which $\tilde f(x) := f(x) + cx^k$ has a root of multiplicity at least~$3$.
	\item
		If $B_{n,k} = 0$, then there is a value of $c\in\CC$ for which $\tilde f(x) := f(x) + cx^k$ has two roots of multiplicity at least~$2$ or a root of multiplicity at least~$4$.
	\end{enumerate}
\end{theorem}
\begin{remark}
  The condition $f(0) \neq 0$ cannot be removed in general, as can be seen in an example such as $f(x) = x^4 + x^2$, $n = 4$, $k = 3$, where $A = \disc(x f'(x) - k f(x)) = \disc(x^4 - x^2) = 0$ although $f(x) + cx^k = x^4 + c x^3 + x^2$ has no triple root for any $c\in\CC$.
\end{remark}
\begin{proof}[Proof of Theorem \ref{thm:triple_double_converse}]
	\begin{enumerate}[(a)]
	\item
		If $k=0$, then $A_{n,k}$ is the discriminant of the polynomial $f'$ of degree $n-1$, so $f'$ has a root $s\in\CC$ of multiplicity at least~$2$. Pick $c\in\CC$ so that $\tilde f(s) = f(s) + c = 0$. Then, $\tilde f(x)$ has a root of multiplicity at least~$3$ at $s$.
		
		If $1\leq k<n$, then $A_{n,k}$ is (up to possibly a factor of $a_0\neq0$ or $a_n\neq0$) the discriminant of the polynomial $x f'(x) - k f(x)$ of degree $n$, so this polynomial has a root $s\in\CC$ of multiplicity at least~$2$. Since $f(0)\neq0$, we have $s\neq0$. We can then choose $c\in\CC$ so that $\tilde f(s) = f(s) + c s^k = 0$. Then $\tilde f(x)$ has a root of multiplicity at least~$3$ at $s$: Recall that $s$ is a root of $x f'(x) - k f(x) = x \tilde f'(x) - k \tilde f(x)$ of multiplicity at least~$2$. Equivalently, $s \tilde f'(s) - k \tilde f(s) = 0$, which implies $\tilde f'(s) = 0$, and $s \tilde f''(s) - (k-1) \tilde f'(s) = 0$, which implies $\tilde f''(s) = 0$.
	\item
		We first discuss the case $k=0$. Let $s_1,\dots,s_{n-1}$ be the roots of $f'(x)$, counted with multiplicity. Since $B_{n,k}=0$, one of the factors $\frac{f(s_i) - f(s_j)}{(s_i - s_j)^3}$ (which, by the proof of Theorem \ref{thm:DDn0_fzn}, is a polynomial in the coefficients of $\frac{f'(x)}{(x-s_i)(x-s_j)}$) must be zero.
		
		If $s_i \neq s_j$, this simply means that $f(s_i) = f(s_j)$. We can therefore choose $c\in\CC$ so that $\tilde f(s_i) = \tilde f(s_j) = 0$. Then, $\tilde f(x)$ has roots of multiplicity at least~$2$ at $s_i$ and $s_j$ since $\tilde f'(s_i) = f'(s_i) = 0$ and $\tilde f'(s_j) = f'(s_j) = 0$.
		
		If $s_i = s_j = s$, then looking at Lemma~\ref{lem:laurent-antiderivative}(a)(e), one sees that the vanishing of the polynomial $\frac{f(s_i) - f(s_j)}{(s_i - s_j)^3}$ is equivalent to $f'(x)$ having a root at $s$ of multiplicity at least~$3$. Choose $c\in\CC$ so that $\tilde f(s) = 0$. Then, $\tilde f(x)$ has a root at $s$ of multiplicity at least~$4$.
		
		The case $1\leq k<n$ works similarly. Since $f(0)\neq0$, the roots $s_1,\dots, s_n$ of the degree $n$ polynomial $x f'(x) - k f(x)$ (counted with multiplicity) are non-zero. They are also the non-zero roots of the derivative of $F_0(x) = f(x) / x^k$. Since $B_{n,k}=0$, one of the factors $\frac{F_0(s_i) - F_0(s_j)}{(s_i-s_j)^3}$ must be zero. Let $F(x) := \tilde f(x)/x^k = (f(x)+cx^k)/x^k = F_0(x) + c$.

		If $s_i\neq s_j$, then $F_0(s_i) = F_0(s_j)$. Choose $c\in\CC$ so that $F(s_i) = F(s_j) = 0$. Then, $F(x)$ has roots of multiplicity at least~$2$ at $s_i$ and $s_j$, and therefore so does $\tilde f(x) = F(x) x^k$.

		If $s_i = s_j$, then Lemma~\ref{lem:laurent-antiderivative}(a)(e) shows that $F_0'(s)$ has a root of multiplicity at least~$3$ at $s$. Choose $c\in\CC$ so that $F(s) = 0$. Then, $F(x)$ has a root of multiplicity at least~$4$ at $s$, and therefore so does $\tilde f(x) = F(x)x^k$.
	\qedhere
	\end{enumerate}
\end{proof}

\begin{remark}
  A natural inquiry is to carry the iteration further to yield a ``triple discriminant,'' etc. The factorizations of Theorem \ref{thm:DD_fzn_intro} show that the discriminant of $DD_{n,k}$ (for $n \geq 3$) with respect to any of its variables is unfortunately $0$. However, it is interesting to form expressions such as $\disc_{a_{\ell}}(A_{n,k})$ or $\Res_{a_\ell}(A_{n,k}, B_{n,k})$. Do these govern coalescences of roots in higher-dimensional families of polynomials $f_0 + a_k x^k + a_\ell x^\ell$?
\end{remark}

\section{The constant factor \texorpdfstring{$c_{n,k}$}{cₙ,ₖ}} \label{sec:const}
\label{section:constant}
The foregoing factorizations take place over $\QQ$. When generalizing to $\ZZ$ or fields of finite characteristic, we must also consider the constant factor $c_{n,k}$, the largest integer dividing $DD_{n,k}$, which we catalog in Table \ref{tab:c} below. For $n \leq 6$ this was done by directly computing $DD_{n,k}$. For $n > 6$ this became impractical, so we tried an assortment of ``compressed'' polynomials $f$ (with certain coefficients set to $0$ or other constants), including over finite and $p$-adic fields, and reported the GCD of the constant factors found. The correct constant $c_{n,k}$ is then guaranteed to divide whatever is written in Table \ref{tab:c}.

\begin{table}[bhtp]
\centering
\begin{tabular}{ c|c|c|c|c|c|c|c|c }
$n\setminus k$ & 0 & 1 & 2 & 3 & 4 & 5 & 6 & 7 \\
\hline
3 & $2^4$ & $2^4$ & & & & & \\
4 & $2^4$ & $2^8$ & $2^8$ & & & & \\
5 & $2^8$ & $2^8$ & $2^8$ & & & & \\
6 & $2^8$ & $2^{12}$ & $2^{12}$ & $2^{12}3^6$ & & & \\ \hdashline
7 & $2^{12^{\vphantom{1^1}}}$ & $2^{12}$ & $2^{12}$ & $2^{12}$ & & & \\
8 & $2^{12}$ & $2^{16}$ & $2^{16}$ & $2^{16}$ & $2^{28}$ & & \\
9 & $2^{16}$ & $2^{16}$ & $2^{16}$ & $2^{16}3^6$ & $2^{16}$ & & \\
10 & $2^{16}$ & $2^{20}$ & $2^{20}$ & $2^{20}$ & $2^{28}$ & $2^{20}5^{10}$ & \\
11 & $2^{20}$ & $2^{20}$ & $2^{20}$ & $2^{20}$ & $2^{20}$ & $2^{20}$ & \\
12 & $2^{20}$ & $2^{24}$ & $2^{24}$ & $2^{24}3^6$ & $2^{36}$ & $2^{24}$ & $2^{32}3^{12}$ \\
13 & $2^{24}$ & $2^{24}$ & $2^{24}$ & $2^{24}$ & $2^{24}$ & $2^{24}$ & $2^{24}$ \\
14 & $2^{24}$ & $2^{28}$ & $2^{28}$ & $2^{28}$ & $2^{40}$ & $2^{28}$ & $2^{36}$ & $2^{28}7^{14}$ \\
15 & $2^{28}$ & $2^{28}$ & $2^{28}$ & $2^{28}3^{6}$ & $2^{28}$ & $2^{28}5^{10}$ & $2^{28}3^{12}$ & $2^{28}$
\end{tabular}
\caption{The outlying constant $c_{n,k}$ of the double discriminant $DD_{n,k}$. For $n > 6$ the values of $c_{n,k}$ are unverified but must divide those given here.}
\label{tab:c}
\end{table}

In view of the orderliness of the values computed, we are inclined to believe that the constants computed in Table \ref{tab:c} are the correct ones. In particular, we observe the following intriguing patterns:
\begin{samepage}
\begin{conj}~\label{conj:c_full}
  \begin{enumerate}[(a)]
    \item The constant $c_{n,k}$ has the same prime divisors as $2 \gcd(k,n)$ if $k > 0$, while $c_{n,0} = 2^{4\floor{(n-1)/2}}$.
    \item $\nu_2(c_{n,k}) \geq 4\left\lfloor\frac{n-1}2\right\rfloor$, with equality when $k = 0$.
    \item For all $n$, $k$, and $p$, $\nu_p(c_{n,k})$ is a multiple of $2p$.
  \end{enumerate}
\end{conj}
\end{samepage}

\begin{remark}
Note that no prime $p > n$ can divide $DD_{n,k}$ as can be seen via Lemma \ref{lem:disc-zero} and Corollary \ref{cor:DDnk F0}.   
\end{remark}

We have shown modest progress toward Conjecture \ref{conj:c_full}.

\begin{prop} \label{prop:2^n}
    $2^{n-1}\mid DD_{n,k}$. If $1 \leq k \leq n/2$, then $2^n\mid DD_{n,k}$.
\end{prop}

\begin{proof}
  By a suitable form of Stickelberger's theorem on discriminants, there exist polynomials $P, Q \in \ZZ[a_0,\ldots,a_n]$ such that $D_n^2 = P^2 + 4Q$. (We may take $P$ to be the ``discriminant Pfaffian'': see \cite[Theorem 4.2]{ABV23}). Now
  \begin{align*}
    \pm DD_{n,k} &= \Res\left(D_n, \frac{\partial}{\partial a_k}D_n\right) \\
    &= \Res\left(D_n, 2 P \frac{\partial}{\partial a_k} P + 4 \frac{\partial}{\partial a_k} Q \right) \\
    &= 2^{\deg_{a_k} D_n} \Res\left(D_n, P \frac{\partial}{\partial a_k} P + 2 \frac{\partial}{\partial a_k} Q \right).
  \end{align*}
  Substituting the value of $\deg_{a_k} D_n$ from Proposition \ref{prop:ldg} yields the result.
\end{proof}

\bibliography{DD_custom}

\begin{thebibliography}{ABV23}

\bibitem[ABV23]{ABV23}
Asher Auel, Owen Biesel, and John Voight.
\newblock Stickelberger's discriminant theorem for algebras.
\newblock {\em Amer. Math. Monthly}, 130(7):656--670, 2023.

\bibitem[Bha23]{Bhargava_vdW_short}
Manjul Bhargava.
\newblock A proof of van der {W}aerden's conjecture on random {G}alois groups
  of polynomials.
\newblock {\em Pure Appl. Math. Q.}, 19(1):45--60, 2023.

\bibitem[Bha25]{Bhargava_vdW}
Manjul Bhargava.
\newblock Galois groups of random integer polynomials and van der {W}aerden's
  conjecture.
\newblock {\em Ann. of Math. (2)}, 201(2):339--377, March 2025.

\bibitem[BM09]{BM09}
Laurent Bus\'{e} and Bernard Mourrain.
\newblock Explicit factors of some iterated resultants and discriminants.
\newblock {\em Math. Comp.}, 78(265):345--386, 2009.

\bibitem[GKZ08]{GKZ}
I.~M. Gelfand, M.~M. Kapranov, and A.~V. Zelevinsky.
\newblock {\em Discriminants, resultants and multidimensional determinants}.
\newblock Modern Birkh\"{a}user Classics. Birkh\"{a}user Boston, Inc., Boston,
  MA, 2008.
\newblock Reprint of the 1994 edition.

\bibitem[Han16]{Han16}
Jing~Jun Han.
\newblock Multivariate discriminant and iterated resultant.
\newblock {\em Acta Math. Sin. (Engl. Ser.)}, 32(6):659--667, 2016.

\bibitem[Hen69]{Henrici1866}
Olaus M. F.~E. Henrici.
\newblock On certain formula concerning the theory of discriminants; with
  applications to discriminants of discriminants, and to the theory of polar
  curves.
\newblock {\em Proc. Lond. Math. Soc.}, 2:104--116, 1866/69.

\bibitem[LM09]{LM09}
Daniel Lazard and Scott McCallum.
\newblock Iterated discriminants.
\newblock {\em J. Symbolic Comput.}, 44(9):1176--1193, 2009.

\end{thebibliography}
\bibliographystyle{alpha}
\end{document}